\documentclass[12pt]{amsart}

\usepackage{latexsym,cite,amsthm,amsmath,amssymb}

\usepackage{srcltx}
\usepackage[cp1251]{inputenc}
\usepackage[T2A]{fontenc}

\date{}

\def\l{\lambda}

\def\a{\alpha}
\def\b{\beta}

\def\e{\varepsilon}
\def\ncja{noncommutative Jordan algebra}

\def\0{\bar 0}
\def\1{\bar 1}

\def\M{\mathcal M}
\def\V{\mathcal V}

\def\ctd{\hfill$\Box$}

\def\OO{\mathbb O}

\def\bes{\begin{eqnarray*}}
\def\ees{\end{eqnarray*}}
\def\bee{\begin{eqnarray}}
\def\eee{\end{eqnarray}}

\def\Proof{{\it Proof. }}

\newtheorem{thm}{Theorem}[section]
\newtheorem{lem}{Lemma}[section]

\newtheorem{cor}{Corollary}[section]

\newtheorem{prop}{Proposition}[section]

\theoremstyle{remark}

\theoremstyle{example}
\newtheorem{example}{Example}[section]
\theoremstyle{question}
\newtheorem{question}{Question}[section]

\vspace {5 mm}

\begin{document}

\title[]{Nonmatrix varieties of nonassociative algebras}
\author[I.\,P.\,Shestakov]{Ivan P.\,Shestakov}\address{Instituto de Matem\'atica e Estat\'istica,
Universidade de S\~ao Paulo,
S\~ao Paulo, Brasil
{\tiny and}
Sobolev Institute of Mathematics, Novosibirsk, Russia}
\email{shestak@ime.usp.br}
\author[V.\,Bittencourt]{Vinicius Souza Bittencourt}\address{
Universidade Federal do Oeste da Bahia, Barreiras, Bahia, Brasil
}
\email{vinicius.bittencourt@ufob.edu.br}

\thanks{The first author is partially supported by  FAPESP, Proc.\,  2018/23690-6 and CNPq, Proc.\,304313/2019-0, Brasil.}



  \vspace {5 mm}

\begin{abstract}
A variety of associative algebras is called  {\it nonmatrix } if it does not contain the algebra of  $2\times 2$ matrices over the given field.  Nonmatrix varieties were introduced and studied by V.\,N.\,Latyshev  \cite{Lat, Lat1, Lat2} in relation with the Specht problem. Some characterizations of nonmatrix varieties were obtained in the paper \cite{MPR}.
In the given paper the notion of nonmatrix variety is extended for nonassociative algebras, and the characterization from  \cite{MPR} is generalized for alternative, Jordan, and some other varieties of algebras.
\end{abstract}

\maketitle

\begin{flushright}
To the memory of Viktor Nikolaevich Latyshev 
\end{flushright}

\section{Introduction}
A variety of associative algebras is called  {\it nonmatrix} if it does not contain the algebra of  $2\times 2$ matrices over the given field.  
Nonmatrix varieties were introduced and studied by V.\,N.\,Latyshev \cite{Lat, Lat1, Lat2} in connection with the Specht problem. Some characterizations of nonmatrix varieties were obtained in the papers   \cite{Kem, MPR}. More exactly, the following result was proved in  \cite{MPR}.

Let $\V$ be a variety of associative algebras over  an infinite field. The following conditions are equivalent:
\begin{enumerate}
\item The variety $\V$ is nonmatrix;
\item every algebra $A\in\V$ satisfies an identity of the type  $[x,y]^m=0$;
\item  every finitely generated algebra $A\in\V$ satisfies an identity of the type  $[x_1,x_2]\cdots [x_{2s-1},x_{2s}]=0$;
\item let $A\in\V$, then for any nilpotent elements     $a,b\in A$ the element  $a+b$ is nilpotent as well;
\item let $A\in\V$, then the set of nilpotent elements forms an ideal in  $A$;
\item let $A\in\V$, then any nilpotent element in $A$ generates a nilpotent ideal in $A$;
\item let $A\in\V$, then every finite set of algebraic elements generates a finite dimensional subalgebra in  $A$.
\end{enumerate}

In this paper we extend the notion of nonmatrix variety to nonassociative algebras and generalize the above characterization to alternative, Jordan, and some other varieties. The most general varieties that have been considered here are the so called  {\em admissible} and {\em locally admissible} varieties of algebras, which were  introduced by the first author in  \cite{Sh1971}. A variety  $\V$ of noncommutative Jordan algebras is called {\em (locally) admissible}, if every anticommutative algebra from  $\V$ is (locally) nilpotent. In particular, the varieties of alternative and Jordan algebras are admissible.We prove some structure theorems for locally admissible varieties and characterize nonmatrix (locally) admissible varieties of algebras over a field of  chatracteristic 0.

\section{Alternative and Jordan algebras}
We will consider algebras over an infinite field  $F$. In particular, all the considered varieties are homogeneous. 

A variety $\V$ of alternative (Jordan) algebras we will call {\em nonmatrix}, if $\V$ does not contain the algebra of  $2\times 2$ matrices $M_2=M_2(F)$ over the field $F$ (respectively, the algebra of symmetric  $2\times 2$ matrices $H_2=H(M_2(F))$ over $F$).
\begin{prop}\label{prop2.1}
For any variety  $\V$ of alternative or Jordan algebras the following conditions are equivalent:
\begin{itemize}
\item [(A)] the variety $\V$ is nonmatrix;
\item  [(B)] every simple algebra  $A\in\V$  is a field;
\item [(C)] every prime algebra $A\in\V$ without nonzero nil ideals is associative and commutative.
\end{itemize}

\end{prop}
\Proof It is clear that in any of the cases  $(A) - (C)$ the variety  $\V$ does not contain free associative algebra in the alternative case and does not contain  free special Jordan algebra in the Jordan case. Therefore, in all the cases every algebra from   $\V$ is a  $PI$-algebra. Since the variety  $\V$ is homogeneous, without loss of generality we may assume that the field  $F$ is algebraically closed.

$(A)\Rightarrow (B)$.
In the case of alternative algebras, by Kaplansky's and Kleinfed's theorems   \cite{ZSSS}, every simple algebra $A\in\V$ is a matrix algebra of order  $n$  or a split Cayley-Dickson algebra  $\OO$ over $F$. Since  $\V$ is nonmatrix and  $M_2\subset \OO$, the only possible case here is  $A=F$.

In the case of Jordan algebras, by Zelmanov's theorem  \cite{Zel}, every simple algebra $J\in\V$ is either matrix algebra $M_n^{(+)}$,  or the algebra of symmetric matrices $H(M_n)$, or the exceptional algebra $H(\OO_3)$ of $3\times 3$ hermitian matrices over the split Cayley-Dickson algebra $\OO$,  or the algebra of bilinear form $J(V,f)$. It is easy to see that all these algebras contain the algebra $H_2$, except the algebra $M_1(F)^{(+)}=F$. Since  $\V$ is nonmatrix, we have $J=F$.
 
 \smallskip
 $(B)\Rightarrow (A)$. It is evident.
 
  \smallskip
$(B)\Rightarrow (C)$. 
By the theorems of Posner, Slater \cite{ZSSS} and Zelmanov  \cite{Zel}, every prime  nil-semisimple  $PI$-algebra $A$ is a central order in a simple algebra, which by (B)  is a field. Hence $A$ is associative and commutative.

 \smallskip
$(C)\Rightarrow (B)$. 
By the theorems of Kaplansky, Kleinfeld \cite{ZSSS} and Zelmanov  \cite{Zel}, a simple  $PI$-algebra from  $\V$ can not be a nil algebra, therefore it is associative and commutative and is a field. 

\ctd

 \begin{thm}\label{thm2.1}
1.  Let $\V$ be a variety of alternative algebras over an infinite field. Then conditions  (1) - (7) from the Introduction are equivalent, where in condition  (3) the arrangements of parenthesis in the product is arbitrary.

2. Let $\V$ be  a variety of Jordan algebras over an infinite field. Then conditions  (1), (4) - (7)  and the conditions (2'), (3') below are equivalent.

\begin{enumerate}
\item [(2')] every algebra $A\in\V$ satisfies an identity of the type   $(x,y,z)^m=0$;
\item  [(3')] every finitely generated algebra $A\in\V$ satisfies an identity of the type    $(x_1,x_2,x_3)\cdots (x_{3s-2},x_{3s-1},x_{3s})=0$, where the arrangements of parenthesis in the product is arbitrary.
\end{enumerate}
\end{thm}
\Proof
$(1)\Rightarrow (2), (2')$.  Let $\V$ be a nonmatrix variety of alternative or Jordan algebras, and let $A$ be the  $\V$-free algebra on generators $x,y,z$. Due to \cite[���. 67]{ZS}, the quotient algebra $A/Nil\,A$ is isomorphic to a subdirect sum of nil-semisimple prime algebras.  
By condition (C) of proposition \ref{prop2.1}, every nil-semisimple prime algebra from $\V$ is associative and commutative, hence so is the quotient algebra $A/Nil\,A$. Therefore, $[x,y],\,(x,y,z)\in Nil\,A$.  SInce $A$ is a free algebra, this proves  (2) and (2').

\smallskip
$(1)\Rightarrow (3), (3')$. Let $A$ be a finitely generated algebra from a nonmatrix variety $\V$ of alternative or Jordan algebras. As above, the elements $[a,b],\,(a,b,c)\in Nil\,A$ for any $a,b,c\in A$. Since $A$ is a $PI$-algebra, by the theorems of Shestakov \cite{Sh1983} and Medvedev \cite{Med}, the radical $Nil\,A$ is nilpotent. This proves (3) and (3').

\smallskip
$(3)\Rightarrow (2), (3')\Rightarrow (2')$. It is evident.

\smallskip
$(2)\Rightarrow (1), (2')\Rightarrow (1)$.
It is easy to see that the matrix algebra $M_2$ does not satisfy the identity $[x,y]^n=0$, hence this identity defines a nonmatrix variety. Similarly, the algebra $H_2$ does not satisfy the identity $(x,y,z)^m=0$, hence this identity defines a nonmatrix variety of Jordan algebras.

\smallskip
The equivalence of conditions  (1) and (4) - (7) will be proved in theorem \ref{thm3.5}  in a more general situation, hence we omit these proofs here.

\ctd

\section{Noncommutative Jordan algebras}

Recall that an algebra $A$ is called a {\em noncommutative Jordan algebra} if it satisfies the identities
\bee
(x,y,x)&=&0,\label{flex}\\
(x^2,y,x)&=&0.\label{jord}
\eee
Identity \eqref{flex} is called  a {\em flexibility} identity. It is clear that any commutative or anticommutative algebra satisfies this identity. A flexible algebra  $A$  satisfies Jordan identity \eqref{jord} if and only if the adjoint commutative algebra  $A^{(+)}$ is a Jordan algebra. Notice that in unital algebras the Jordan identity  \eqref{jord} implies the flexibility identity \eqref{flex}.

The class of noncommutative Jordan algebras is extremally large. It contains all Jordan and alternative algebras, as well as all anticommutative algebras. Let us give more examples of noncommutative Jordan algebras.

\smallskip
\begin{example} Quadratic flexible algebras.
\end{example}
A unital algebra $A$ over a field $F$ is called {\em quadratic over  $F$}, if for any element  $a\in A$  there exist elements  $t(a),n(a)\in F$ such that 
\bes
a^2-t(a)a+n(a)1=0.
\ees
It is easy to see that any quadratic flexible algebra satisfies Jordan identity  \eqref{jord}, hence it is a noncommutative Jordan algebra.  An important class of such algebras is formed by  simple central algebras of dimension  $2^n$, obtained by the Cayley-Dickson process \cite[Chapter 2]{ZSSS}. For any  quadratic simple central flexible algebra  $A$ of dimension more then 2 the adjoint symmetric algebra $A^{(+)}$ is a simple Jordan algebra of bilinear form \cite{Smith}.

\begin{example} Quasi-associative algebras.
\end{example}
An algebra $A$ over a field  $F$ is called {\em quasi-associative}, if there exists a quadratic extension  $K$ of the field  $F$ such that the algebra $A_K=K\otimes_FA$ is isomorphic to the algebra  $B(\l)$, defined as follows:
$ B$ is an associative algebra over  $K$, $\l\neq\tfrac12$ is a fixed element of the field  $K$, and  $B(\l)$ has the same vector space over  $K$ as $B$, but with the product  $x*y = \l xy + (1-\l)yx$, where $xy$ is a product in  $B$.

Every quasi-associative algebra is a noncommutative Jordan algebra. Moreover, it is proved in \cite{Ded} that an algebra $A$ is quasi-associative if and onlyb if it satisfies the identity 
\bes\label{Id_qass}
(x,y,z)=\a[[x,z],y],
\ees
for some $\a\in F,\,\a\neq 1/4$.

\begin{example}
 Kokoris algebras \cite{Kok58}.
 \end{example}
Let $A$ be an algebra. Recall (see \cite[Lemma 3.1]{ZSSS}) that the associator $(a,b,c)_+$  
in the algebra $A^{(+)}$ is represented via operations in  $A$ as follows:
\bee\label{id_as+}
4(a, b, c)_+ &=& (a, b, c) - (c, b, a) + (b, a, c) + (a, c, b)\\
&-& (c, a, b) - (b, c, a) + [b, [a, c ] ].\nonumber
\eee
If  $A$ is flexible then the right part of this equality may be written as
 $$
 4(a, b, c) - J(a, b, c)-[[a, c],b ],
$$ 
where
\bes
J(a, b, c)&=&(a,b,c)+(b,c,a)+(c,a,b)-(b,a,c)-(a,c,b)-(c,a,b)\\
&=&[[a,b],c]+[[b,c],a]+[[c,a],b].
\ees
Therefore, the following proposition holds:
\begin{prop}\label{prop3.0}
For a flexible algebra $A$, the algebra $A^{(+)}$ is associative if and only if   $A$ satisfies the identity
\bee\label{id_Kokoris}
J(a, b, c)= 4(a, b, c) -[[a, c],b ].
\eee
\end{prop}

 We will call an algebra $A$ {\em a Kokoris algebra}, if it satisfies identity \eqref{id_Kokoris}. Clearly, a Kokoris algebra is flexible. Moreover, since the adjoint commutative algebra  $A^{(+)}$  is associative for a Kokoris algebra $A$, the algebra $A$ is also noncommutative Jordan.  
 
Observe that, for a Kokoris algebra  $A$, the algebra  $A^{(+)}$  with the bracket $\{x,y\}=\tfrac12[x,y]=\tfrac12(xy-yx)$ is a {\em generis Poisson algebra}. Recall (see \cite{Sh1998, KMLS}),
 that an algebra $A$ with the multiplications  $\cdot,\,\{,\}$ is called a  {\em generic Poisson algebra} if
\begin{itemize}
\item the multiplication $\cdot$ is associative and commutative;
\item the bracket $\{,\}$ is anticommutative;
\item  the Leibniz identity  $\{x\cdot y,z\}=x\cdot\{y,z\}+y\cdot\{x,z\}$ holds.
\end{itemize}
Recipricaly, if  $\langle A,\cdot,\{,\},+\rangle$ is a generic Poisson algebra then the algebra $\langle A,*,+\rangle$ with the multiplication  
$a*b=a\cdot b+\{a,b\}$ is a Kokoris algebra. 

Classical Poisson algebras relate, in this correspondence, to  {\em Lie admissible} Kokoris algebras, that is, the algebras satisfying the identity $4(x,y,z)=[[x,z],y]$.

L.\,Kokoris in \cite{Kok58} defined these algebras in the case when   $A=F[X]$ is the polynomial algebra on a set of variables    $X =\{x_j\,|\,j\in\Lambda\}$, with the bracket
 \bes
\{f,g\} = \sum_{i<j} c_{ij}(\tfrac{\partial f}{\partial x_i}\tfrac{\partial g}{\partial x_j}-\tfrac{\partial f}{\partial x_j}\tfrac{\partial g}{\partial x_i}),
\ees
where $\{c_{ij}\,|\,i,j\in\Lambda\}$ is a family of elements from $F[X]$ such that $c_{ij} =-c_{ji}$  for all $i, j\in\Lambda$. 

\smallskip

We will need the following fundamental result by V.\,G.\,Skosyrsky \cite{Skos,  Skos1} on description of strongly prime noncommutative Jordan algebras. Recall that a prime noncommutative Jordan algebra  $A$ is called  {\em strongly prime} if it does not contain nonzero ideals  $I$ such that   $I^{(+)}$ is contained in the McCrimmon radical  ${\mathcal M}(A^{(+)})$ \cite[14.1]{ZSSS} of the adjoint Jordan algebra  $A^{(+)}$. In particular, if  $A$ is prime and  $Nil\,A=0$ then  $A$ is strongly prime. 

\begin{thm}\label{thm3.0} (\cite{Skos})
Let $A$ be a strongly prime   \ncja, then one of the following conditions holds: 
	\begin{itemize}
		\item[(I)] $A$ is a commutative prime nondegenerate Jordan algebra.
		
		\item[(II)] $A$ is a central order in a prime quasi-associative algebra over the extended centroid of the algebra $A$.
		
		\item[(III)] $A$ is a central order in a simple quadratic flexible algebra.
		
		\item[(IV)] $A$ is a Kokoris algebra.
		
	\end{itemize}
\end{thm}

As we have already seen, the class of noncommutative Jordan algebras is too large.  To obtain meaningful structural results in this class, additional restrictions are needed. 

 \smallskip
A variety $\V$ of noncommutative Jordan algebras we will call {\em (locally) admissible,} if every anticommutative algebra $A\in\V$ is  (locally) nilpotent. This condition is a very particular case of the following result valid in alternative and Jordan algebras: every nil algebra of bounded degree is locally nilpotent.

Consider in a noncommutative Jordan algebra $A$ the subspace  $I_2(A)=(A^{(+)})^2$.  It was proved in  \cite{Sh1971} that  $I_2(A)$ is an ideal in  $A$, and the quotient algebra $A/I_2(A)$ is anticommutative.  Therefore, a variety  $V$ is admissible  (locally admissible), if there exists a natural number  $n$ (respectively, $n(k)$) such that the following inclusions are true in the free algebra of countable rank  $F_{\V}[X]$ (respectively, in the free algebra $F_{\V}[X_k]$ of rank $k$):
\bee
(F_{\V}[X])^n&\subseteq &I_2(F_{\V}[X]),\label{n} \\
(F_{\V}[X_k])^{n(k)}&\subseteq &I_2(F_{\V}[X_k]). \label{n(k)}
\eee
The corresponding numbers $n$ and  $n(k)$ we will call {\em indices of admissibility and of local admissibility } of the variety $\V$.
For example, the varieties of Jordan, associative, and alternative algebras are admissible of indices respectively  2, 3, and  4 (see \cite{ZSSS}).
In  \cite{Sh1971} it was proved that the variety of noncommutative Jordan algebras satisfying the identity  $[x,(z,y,y)]=0$  is admissible of index  7. In the same paper it was proved that he variety of noncommutative Jordan algebras defined by  the identity $([x,y],z,z)=0$ is locally admissible. It is easy to see that an anticommutative algebra with the identity  $((xy)z)z=0$  is solvable but not necessary nilpotent, hence the last variety is not admissible.  

\begin{prop}\label{prop3.01}
A variety  $\V$ of noncommutative Jordan algebras is admissible of index   $m$ if and only if a system of multilinear identities of the following type verifies in  $\V$ 
\bee\label{m-lie}
w(x_1,x_2,\ldots,x_m)=\sum \a_i u_i\circ v_i,
\eee
for all ``purely Lie'' products $w$ of degree $m$ and some $\a_i\in F$  and monomials $u_i,v_i $ depending on $w$.
\end{prop}
\begin{proof}
Clearly, any admissible variety of index  $m$ satisfies the given condition. Assume now that $\V$ satisfies $\eqref{m-lie}$ and prove that $\V$ is $m$-admissible.
Let $U$ be a product of  $m$ elements of the algebra $F_{\V}[X]$. Replace all the products  $xy$ in  $U$ by the sums $xy=\tfrac12 ([x,y]+x\circ y)$, then $U$ is represented as a sum of monomials  $u_i$ with respect to operations  $[x,y]$ and   $x\circ y$. If such a monomial  $u_i$ contains at least one operation $x\circ y$ then in view of the fact that  $I_2(F_{\V}[X])$ is an ideal, we would have $u_i\in I_2(F_{\V}[X])$. And if $u_i$ does not contain operations $x\circ y$,  then it is a purely Lie product and  $u_i\in I_2(F_{\V}[X])$ by  \eqref{m-lie}.
\end{proof}

\begin{cor}\label{cor3.0}
If $((F_{\V}[X])^{(-)})^m=0$ the the variety $\V$ is admissible of index $m$.
\end{cor}

\smallskip

The idea of the proof of the following lemma goes back to \cite[Lemma 8]{Sh1971}.
\begin{lem}\label{lem3.1}
Let $A\in\V$ where $\V$ is a locally admissible variety. If $A$ is finitely generated then so is the Jordan algebra  $A^{(+)}$.
\end{lem} 
\Proof
 It suffices to consider the case when the algebra  $A$ is free in  $\V$. Let
$A = F_{\V} [x_1,\ldots,x_k]$ be a $\V$-free algebra on generators $x_1,\ldots,x_k$. Then $A^{n(k)}\subset I_2(A)$. Put $Y=\{v\in A\,|\, v \hbox{ is a monomial on } x_i \hbox{ of degree } \deg v \leq n(k)\}$, and let $J$ be a subalgebra of the algebra $A^{(+)}$ generated by the set $Y$.
Let us show that $A^{(+)}=J$. It suffices to prove that any monomial  $v$ (on variables $x_1,\ldots,x_k$) belongs to $J$. Apply induction on  $\deg v$. If $\deg v\leq n(k)$ then $v\in Y\subset J$. 
Let $\deg v>n(k)$, then $v\in I_2(A),\ v=\sum \a_i u_i\circ v_i$. By homogenety of the variety $\V$ we have $\deg u_i+\deg v_i=\deg v,$ hence $\deg u_i,\deg v_i<\deg v$, and by induction  $u_i,v_i\in J$.
Therefore, $v\in J$ and  $A^{(+)}=J$.

\ctd

\begin{lem}\label{lem3.2}
Let $A$ be a finitely generated algebra from a locally admissible variety  $\V$, and  $I$ is an ideal in  $A$. Then there exists a natural number  $m$ such that for any elements  $a_1,\ldots,a_m\in A$ 
\bes
IT_{a_1}\cdots T_{a_m}\subseteq I\circ A, \hbox{ where } \ T_{a_j}\in\{R_{a_j},L_{a_j}\}.
\ees
If  $\V$ is an admissible variety then the statement of the lemma is valid for any algebra $A\in\V$, not necessarily finitely generated.
\end{lem}
\Proof
Let $A$ have $k$ generators. Consider the $\V$-free algebra $F_{\V}[x_1,\ldots,x_k,y]$, and let  $m=n(k+1)-1$ from \eqref{n(k)}. Then for any 
$f_1,\ldots,f_{m}\in F_{\V}[x_1,\ldots,x_k]$ we have
\bee\label{y}
yT_{f_1}\cdots T_{f_m}=\sum u_i\circ v_i, 
\eee
where all $u_i$ contain $y$, and $v_i\in F_{\V}[x_1,\ldots,x_k]$. Let now   $a_1,\ldots,a_m\in A$, then choose elements  $f_1,\ldots,f_m\in F_{\V}[x_1,\ldots,x_k]$, which are pre-images of the elements   $a_1,\ldots,a_m$ under the natural epimorphism   $\pi:F_{\V}[x_1,\ldots,x_k]\rightarrow A$. Let $b$ be an arbitrary element from  $I$; we extend the epimorphism  $\pi$ on the algebra $F_{\V}[x_1,\ldots,x_k,y]$ by setting  $\pi(y)=b$. Then equality \eqref{y}  gives
\bes
bT_{a_1}\cdots T_{a_m}=\sum u_i\circ v_i, \hbox{ where } u_i\in I, v_i\in A.
\ees
This proves the lemma  for locally admissible  $\V$. Notice that the number $m$ here depends only on the number of generators of the algebra  $A$ and does not depend on the ideal $I$.

If $\V$ is an admissible variety then all the arguments remain true for any algebra  $A\in\V$, one has only to take   $m=n-1$  for $n$ from \eqref{n}.

\ctd

The following lemma was proved in  \cite{Sh1971}
\begin{lem}\label{lem3.3} (\cite{Sh1971})
Let $A$  be an arbitrary algebra. Then every element  $v\in A^{2^n}$  can be represented in the form   $a_1T_{a_2}\ldots T_{a_n}$, where   $a_i\in A$  and	 $T_{a_j}\in\{L_{a_j},R_{a_j}\}$.
\end{lem}

\begin{prop}\label{prop3.1}
Let $A$ be a  finitely generated algebra from a  locally admissible variety $\V$. Then for any natural number $n$ there exists a number  $f(n)$ such that  $A^{{f(n)}}\subseteq (A^{(+)})^n$.

If $\V$ is an admissible variety then the statement is true for any algebra  $A\in\V$, not necessarily finitely generated.
\end{prop}
\Proof
Consider the case of locally admissible variety, for an admissible variety the proof is similar. We will use induction on  $n$. Clearly, $f(1)=1,\  f(2)=n(k)$, where $k$ is a number of generators of the algebra $A$ and $n(k)$ is the index of local admissibility of the variety  $\V$ from \eqref{n(k)}. Assume that  $f(n-1)$ exists, we prove then that one can take  $f(n)=2^{f(n-1)+m}$, where  $m$ is the number defined in lemma  \ref{lem3.2}.

Let $a\in A^{2^{f(n-1)+m}}$. By lemma \ref{lem3.3}, 
\bes
a=a_1T_{a_2}\ldots T_{a_{f(n-1)}}T_{a_{f(n-1)+1}}\ldots T_{a_{f(n-1)+m}}, \hbox{ where } \ a_j\in A,\ T_{a_j}\in\{R_{a_j},L_{a_j}\}.
\ees
By the induction assumption, $b=a_1T_{a_2}\ldots T_{a_{f(n-1)}}\in A^{f(n-1)}\subseteq (A^{(+)})^{n-1}$. Applying now lemma   \ref{lem3.2} for the ideal $I=A^{f(n-1)}$, we get
 \bes
 bT_{a_{f(n-1)+1}}\ldots T_{a_{f(n-1)+m}}\in I\circ A=A^{f(n-1)}\circ A\subseteq (A^{(+)})^{n-1}\circ A\\
 =(A^{(+)})^{n-1}\circ A^{(+)}\subseteq (A^{(+)})^n.
 \ees
The proposition is proved.

\ctd
\begin{cor}\label{cor3.1}
Let $A$ be a prime algebra from a locally admissible variety  $\V$.  If $A$ does not contain nonzero locally nilpotent ideals then  $A$ is strongly prime. In particular, every simple algebra from   $\V$ is strongly prime.
\end{cor}
\Proof Assume that  $A$  is not strongly prime, then  $A$ contains a nonzero ideal  $I$ such that $I^{(+)}\subset{\mathcal M}(A^{(+)})$, where ${\mathcal M}(A^{(+)})$ is the McCrimmon radical of the algebra $A^{(+)}$. Due to  \cite{Zel}, the McCrimmon radical is locally nilpotent, hence the ideal  $I^{(+)}$ is locally nilpotent. Let $B$ be a finitely generated subalgebra of  $I$, then by lemma  \ref{lem3.1} the algebra $B^{(+)}$ is finitely generated and henceforth is nilpotent. Let $(B^{(+)})^n=0$, then by proposition \ref{prop3.1} we have  $B^{f(n)}=0$. Thus the ideal  $I$ is locally nilpotent, a contradiction. Therefore, the algebra  $A$ is strongly prime.

\ctd

It was proved in  \cite{Sh1971}  that there exists a locally nilpotent radical in any admissible variety of noncommutative Jordan algebras. 
The above results admit to extend this fact to locally admissible varieties. Moreover, we will prove more general result that the property of local finiteness in  Shirshov`s sense  is radical in any locally admissible variety of algebras.

Recall the definition. Let  $A$ be an algebra over an associative commutative ring  $\Phi$, and let $Z$ be an ideal of the ring 
$\Phi$. A finitely generated algebra $A$ is called  {\em finite over $Z$} (in Shirshov's sense), if for some natural number  $m$ there exist elements  $a_1,\ldots,a_n\in A$ such that $A^m\subset Za_1+\cdots +Za_n$. If every finitely generated subalgebra of the algebra $A$ is finite over  $Z$, then $A$ is called  {\em locally finite over $Z$}.
  
\begin{thm}\label{thm3.1}
Let $\V$ be a locally admissible variety of noncommutative Jordan algebras over an associative commutative ring  $\Phi$, and let  $Z$ be an ideal of the ring $\Phi$. Then the locally finite over $Z$ radical  exists in  $\V$.
\end{thm}
\Proof
Following the standard scheme (see  \cite{Sh1971, ZSSS}), it suffices to prove that if  $A$ is a finitely generated algebra from  $\V$ which contains a locally finite over  $Z$ ideal   $I$ such that the quotient algebra $A/I$ is finite over $Z$, then  $A$ is finite over $Z$. Consider the adjoint Jordan algebra $A^{(+)}$. By lemma  \ref{lem3.1}, it is finitely generated. The ideal $I^{(+)}$ is locally finite over  $Z$  in $A^{(+)}$, and the quotient algebra  $A^{(+)}/I^{(+)}\cong (A/I)^{(+)}$  is finite over  $Z$. Since the locally finite over $Z$ radical exists in the class of Jordan algebras (see \cite{ZS}), the algebra $A^{(+)}$ is finite over $Z$. Let  $(A^{(+)})^m\subset Za_1+\cdots+Za_n$ for some  $a_1,\ldots,a_n\in A$, then due to proposition   \ref{prop3.1} we have  $A^{f(m)}\subseteq (A^{(+)})^m\subset Za_1+\cdots+Za_n$, that is  $A$ is finite over $Z$.

\ctd

\begin{cor}\label{cor3.2}
In any locally admissible variety of algebras there exist locally nilpotent and locally finite (over a field) radicals.
\end{cor}

Let us present two more interesting results for (locally) admissible varieties.

\begin{thm}\label{thm3.2}
Every nil algebra of bounded index from a locally admissible variety  $\V$ is locally nilpotent.
\end{thm}
\Proof
Let $A$ be a finitely generated nil algebra of index  $n$ from $\V$. Then the Jordan algebra  $A^{(+)}$ is also a nil algebra of index $n$. By lemma  \ref{lem3.1}, the algebra  $A^{(+)}$ is finitely generated, hence by Zelmanov's theorem \cite{Zel}  it is nilpotent. Let $(A^{(+)})^n=0$, then by proposition  \ref{prop3.1} we have  $A^{f(n)}\subseteq (A^{(+)})^n=0$.

\ctd

\begin{cor}\label{cor3.21}
Every finite dimensional nil algebra  from a locally admissible variety  $\V$ is nilpotent.
\end{cor}

\begin{thm}\label{thm3.3}
Let  $\V$ be a locally admissible variety, and let  $A$  be a finitely generated algebra from $\V$. If the adjoint Jordan algebra  $A^{(+)}$  is a  $PI$-algebra, then $(Nil\, A)^{(+)}$ is a nilpotent ideal in the algebra  $A^{(+)}$.   If the variety $\V$  is admissible then in the assumptions of the theorem the nil radical  $Nil\,A$ is nilpotent.
\end{thm}
\Proof
By lemma \ref{lem3.1}, the Jordan algebra $A^{(+)}$ is finitely generated, hence by Shestakov-Medvedev theorem \cite{Sh1983, Med} the nil radical  $Nil\, A^{(+)}$ is nilpotent.  
Let $(Nil\, A^{(+)})^n=0$. It is clear that $(Nil\, A)^{(+)}\subseteq Nil\, A^{(+)}$, hence $( (Nil\, A)^{(+)})^n \subseteq (Nil\, A^{(+)})^n=0.$

If the variety  $\V$ is admissible, then by proposition \ref{prop3.1} we get $(Nil\, A)^{f(n)}\subseteq( (Nil\, A)^{(+)})^n =0$.

\ctd

It remains an open the question whether the nil radical $Nil\,A$ is nilpotent when  the variety $\V$ is locally admissible?

\medskip

Recall that a variety $\V$ is called  {\em unitary closed} if for any algebra  $A\in\V$,  the algebra $A^{\sharp}$ obtained by adjoining the external unit element to  $A$, belongs to $\V$ as well. 

\begin{thm}\label{thm3.4}
Every strongly prime Kokoris algebra $A$  over a field of characteristic  0 from a locally admissible unitary closed variety  $\V$ is associative and commutative.
\end{thm}
\Proof
In view of associativity of  $A^{(+)}$, it suffices to prove that the algebra  $A$ is commutative. It is known  \cite{Sl} that the nil radical  $Nil(J)$ of any Jordan algebra  $J$ over a field of characteristic zero is invariant under derivations. In our case this implies that  $[Nil(A^{(+)},A]\subset Nil(A^{(+)}$, hence  $Nil (A^{(+)}) $ is an ideal in $A$ and  it coincides with   $Nil (A)$. Since the algebra $A^{(+)}$ is associative, the McCrimmon radical  $\mathcal{M}(A^{(+)})$ is equal to $Nil(A^{(+)})$  and coincides with the set of nilpotent elements of the algebra  $A$. Therefore, the algebra $A$ has no nilpotent elements, and it suffices to prove that every commutator $[x,y]$ is nilpotent.

Let $F_{\V}[x,y]$  be the free  2-generated algebra from  $\V$, then by definition of a locally admissible variety,   $(F_{\V}[x,y])^{n(2)}\subset I_2(F_{\V}]x,y]).$ In particular, an identity of the following form is true in  $\V$:
\[ \ 
[\cdots[x,\underbrace{y],\ldots,y}_{\text{m}}]=\sum \a_i u_i\circ v_i,
\]
for some  $\a_i\in F$  and monomials  $u_i,v_i $ on $x,y$ of full degree $m$ on  $y$ and degree 1 on  $x$. Replacing each multiplication $uv$ on the right side of this identity by $\tfrac12 ([u,v]+u\circ v)$, due to the associativity and commutativity of the algebra $A^{(+)}$, in the algebra $A$ we get the identity
\[ \ 
f(x,y)=[\cdots[x,\underbrace{y],\ldots,y}_{\text{m}}]-\sum_{i< m}\l_i [\cdots[x,\underbrace{y],\ldots,y}_{\text{i}}] \circ y^{m-i}=0
\]
Let $i$ be the minimal index for which  $\l_i\neq 0$. Since the variety $\V$ is unitarily closed, $A$ satisfies the identity
\[
0=\Delta_y^{m-i}(f)=-\l_i[\cdots[x,\underbrace{y],\ldots,y}_{\text{i}}],
\]
where $\Delta^i_y$ is the operator of ``partial derivation '' on $y$ (see \cite[1.6]{ZSSS}).  Assume that  $i>1$, and let  $z=[\cdots[x,\underbrace{y],\ldots,y}_{\text{i-1}}]\neq 0$, Write $z=[u,y]$, then $[[u,y],y]=0$, and we have
\[
0=[\cdots[u^i,\underbrace{y],\ldots,y}_{\text{i}}]=i![u,y]^i,
\]
which implies  $z=0$, a contradiction. Hence $i=1$, and the algebra $A$ is commutative.

\ctd

It remains an open question whether the theorem is true for Kokoris algebras of positive characteristic?

\begin{prop}\label{prop3.2}
Let  $\V$ be a locally admissible variety of noncommutative Jordan algebras. The following conditions for $\V$ are equivalent:
\begin{enumerate}
\item Every strongly prime algebra $A\in \V$  is associative and commutative.
\item Every simple algebra $A\in\V$ is a field. 
\item The variety  $\V$ does not contain noncommutative simple quadratic algebras and noncommutative strongly prime Kokoris algebras.
\end{enumerate}
\end{prop}
\Proof
$(1)\Rightarrow (2)$. 
Let $A$ be a simple algebra from $\V$. If   $A$ is not nil then  $A$ is strongly prime and therefore is associative and commutative. Since   $A$ is simple, it is a field.
Assume now that   $A$ is a nil algebra. If $A^{(+)}\neq \M(A^{(+)})$ then  $A$ is still strongly prime and hence it is associative and commutative. But in this case  $A$  can not be at the same time simple and nil, a contradiction. Now if $A^{(+)}=\M(A^{(+)})$ then  $A^{(+)}$ is locally nilpotent.  Let  $B$ be a finitely generated subalgebra of  $A$, then by lemma \ref{lem3.1}  the algebra $B^{(+)}$ is also finitely generated and therefore is nilpotent. By proposition \ref{prop3.1}, the algebra $B$ is nilpotent as well. Therefore, the algebra  $A$ is locally nilpotent, which is impossible in view of  \cite[Proposition 7.1]{ZSSS}. 

$(2)\Rightarrow (3)$. It is evident.

$(3)\Rightarrow (1)$. 
Let $A$ be a strongly prime algebra, then by theorem   \ref{thm3.0} $A$ is an algebra of types  (I) -- (IV). In view of the assumption, it suffices to consider the algebras of types  (I) -- (III).

Consider first the case when   $A$ is a Jordan algebra. Observe that the algebra $H(F_2)$ is simple and quadratic, hence  $H(F_2)\notin\V$  and  $H(F_2)\notin Var\,A$. Therefore, the variety $Var\,A$ is a nonmatrix variety of Jordan algebras.  Now  by proposition \ref{prop2.1}, (C),  $A$ is  associative.

Let  $A$ be a central order in a prime quasi-associative algebra. Without less of generality, we may assume that $A$ itself is prime quasi-associative, that is,  $A_K\cong B^{(\l)}$, where  $K$ is a quadratic extension of   $F,\ \l\in K$, and  $B$ is an associative algebra. It is easy to see (see, for instance, \cite{Ded}) that $B=(A_K)^{(\mu)}$, where $\mu=\frac{\l}{2\l-1}$. Consider the class of algebras $\V^{(\mu)}=\{C^{(\mu)}\,|\,C\in\V\}$ which by \cite{Ded} is a variety. Since $B\in \V^{(\mu)}$ then the variety $Var\,B\subset \V^{(\mu)}$. 
Consider the quasi-associative algebra  $M_2(F)^{(\l)}$. Clearly, it is simple, quadratic, and noncommutative, hence  $M_2(F)^{(\l)}\notin\V$. But then  $M_2(F)=(M_2(F)^{(\l)})^{(\mu)}\notin \V^{(\mu)}$, and moreover $M_2(F)\notin Var\,B$.
Therefore, $ Var\,B$ is a nonmatrix variety of associative algebras. 
It is easy to see that the algebra  $B$ is prime, hence it is commutative.  But then $A=B$ is commutative and associative.

Finally, let $A$ be a central order in a simple quadratic algebra. Since  $A$ and its central closure generate the same variety,  we may assume that  $A$ is a simple algebra. Then $\dim A\leq 2$, and $A$  is associative and commutative.

\ctd

By  theorem \ref{thm3.4}, we have 
\begin{cor}\label{cor3.31}
In a unitary closed locally admissible variety of algebras $\V$ over a field of characteristic  0 condition (3) is equivalent to the following:\\[1mm]
(3')  {A variety  $\V$ does not contain noncommutative simple quadratic algebras.}
\end{cor}

A locally admissible variety of algebras we will call   {\em nonmatrix}, if it satisfies one of equivalent  conditions (1) -- (3) of proposition  \ref{prop3.2}.

\medskip

\begin{lem}\label{lem3.4}
Let $A$ be a simple quadratic flexible algebra. If  $A$ satisfies the identities  $(x,y,z)^n=[x,y]^m=0$, then  $\dim_F A\leq 2$ and $A$  is commutative and associative.
\end{lem}
\Proof
For any element  $a\in A$ we have
\[
a^2=t(a)a-n(a)1,\]
where $t(a),\, n(a)\in F$ are respectively  {\em the trace} and {\em the norm} of the element $a$.
Recall (see \cite{Osb, Smith}) that $A=F\cdot 1+V$, where $V=\{v\in A\,|\, t(v)=0\}$, with the product
\bes
(\a+v)(\b+u)=(\a\b+(v,u))+(\a u+\b v+v\times u),
\ees
where  $(,):V\times V\rightarrow F$ is a nodegenerate symmetric bilinear form on  $V$, and $u\times v$ is an anticommutative multiplication in   $V$, satisfying the condition  $(u\times v,w)=(u,v\times w)$. In particular, $t([a,b])=0$ in $A$, hence  $[a,b]^2=-n([a,b])\in F$, and since  $[a,b]^m=0$ then  $[a,b]^2=0$ for any $a,b\in A$. 
Linearizing this equality, we get $[a,b]\circ[a,c]=0$ in  $A$, from where for any $u,v,w\in V$ 
\bes
0=[u,v]\circ[u,w]=4(u\times v)\circ (u\times w)=8(u\times v,u\times w)=8((u\times v)\times u, w),
\ees
which in view of the nondegeneracy of the form  $(u,v)$  on $V$ implies 
\bee\label{id_uvv}
(u\times v)\times u=0. 
\eee
Consider the associator
\bes
(u,v,v)&=&(u,v)v+(u\times  v,v) +(u\times v)\times v-(v,v)u\\
&\stackrel{\eqref{id_uvv}}=&(u,v)v-(v,v)u+(u,v\times v)=(u,v)v-(v,v)u.
\ees
Assume that  $\dim_FV>1$  and the elements $u,v$  form a part of an orthonormal  base of  $V$ relatively the form $(u,v)$. 
Then $(u,v,v)=-(v,v)u,\ (u,v,v)^2=(v,v)^2(u,u)$ is a nonzero element of the field  $F$, which contradicts the condition of nilpotence of associators.  Therefore, $\dim_FA<2$, and the algebra $A$ associative and commuative.

\ctd

Now we can describe nonmatrix locally admissible varieties.
\begin{thm}\label{thm3.5}
Let $\V$ be a locally admissible unitary closed variety of noncommutative Jordan algebras over a field of characteristic  $0$. The following conditions are equivalent:
\begin{enumerate}
\item $\V$ is a nonmatrix variety;
\item every algebra $A\in\V$ satisfies the identities    $(x,y,z)^n=0, \ [x,y]^m=0$;
\item  every finitely generated algebra $A\in\V$ satisfies all the identities of the form   $u_1\circ\cdots\circ u_n=0$, where a number $n$ is fixed, all $u_i$ are commutators or associators, and the product is considered in the Jordan algebra  $A^{(+)}$ for an arbitrary arrangement of parentheses;
\item let $A\in\V$, then the set of nilpotent elements forms an ideal in $A$;
\item let $A\in\V$, then every nilpotent element in $A$ generates a nil ideal in $A$;
\item let $A\in\V$, then for any nilpotent elements     $a,b\in A$ the element $a+b$ is nilpotent as well;
\item let $A\in\V$, then any finite set of algebraic elements generates a finite dimensional subalgebra in  $A$.
\end{enumerate}
\end{thm}
\Proof
$(1)\Rightarrow (2)$.  Let  $A$ be the  $\V$-free algebra on generators  $x,y,z$. In view of \cite[p. 67]{ZS}, the quotient algebra $A/Nil\,A$  is isomorphic to a subdirect sum of nil-semisimple prime algebras.  Since every nil-semisimple prime algebra from  $\V$  is strongly prime and therefore is commutative and associative, so is the quotient algebra $A/Nil\,A$. Therefore,, $[x,y],\,(x,y,z)\in Nil\,A$.   Since $A$ is a free algebra, this proves  (2).
\smallskip

$(2)\Rightarrow (1)$.  Let  $A$ be a strongly prime algebra from  $\V$. If $A$ is Jordan then  $A$ is associative by theorem \ref{thm2.1}.  If $A$ is quasi-associative then, as above, we may assume that $A=B^{(\l)},\,\l\neq\tfrac12$, where $B$ is an associative algebra, $B=A^{(\mu)},\,\mu=\frac{\l}{2\l-1}$. It is clear that the ideals in  $A$ and in  $B$ coincide, and  $B$ is a prime algebra. Furthermore, the commutator  $[x,y]_{\l}$ in  $B$ is equal to $[x,y]_{\l}=(2\l-1)[x,y]$, where  $[x,y]$ is the commutator in  $A$, hence  $B$ satisfies the identity $[x,y]^m=0$ as well. Since $B$ is prime, it is commutative,  therefore  $A=B$ is associative and commutative. 

Now theorem \ref{thm3.4} and lemma \ref{lem3.4} imply that the variety  $\V$ is nonmatrix.

\smallskip
$(1)\Rightarrow (3)$. 
Let $A$ be a finitely generated algebra from $\V$.  It was observed in the proof of implication $(1)\Rightarrow (2)$ that $[a,b],\,(a,b,c)\in Nil\,A$ for any $a,b,c\in A$. By \eqref{id_as+}, the jordan associator $(a,b,c)_+\in Nil\,A$, hence there exists  $n$ such that $(a,b,c)_+^n=0$. One may assume that $A$ is a $\V$-free algebra and  that $a,b,c$ are the free generators, then the equality $(a,b,c)_+^n=0$ will be satisfied for any elements from $A$. Therefore, the Jordan algebra   $A^{(+)}$ satisfies the identity $(x,y,z)^n=0$.  In particular,   $A^{(+)}$ is a Jordan  $PI$-algebra, and the needed statement follows from theorem \ref{thm3.3}.
 
 \smallskip
$(3)\Rightarrow (2)$. This is evident.

\smallskip

$(1)\Rightarrow (4)\Leftrightarrow (5)\Rightarrow (6)$.
It follows from the proof of implication  $(1)\Rightarrow (2)$ that the quotient algebra $A/Nil\,A$ is associative and commutative. Since it is nil-semisimple, it does not contain nilpotent  elements. In particular, the nil radical  $Nil\,A$ coincides with the set of nilpotent elements of the algebra $A$, which implies the needed statements.
 
 \smallskip
$(6)\Rightarrow (1)$. In view of proposition  \ref{prop3.2}, it suffices to prove that the variety  $\V$ does not contain simple quadratic algebras of dimension more than  $2$.  Assume that $\V$ contains such an algebra $A=F\cdot 1+V,\,\dim_FV>1$. Without less of generality, the field  $F$ may be assumed algebraically closed. Choose elements  $u,v$ from an orthonormal base of  $V$ relatively the bilinear form  
$(u,v)$ (see the proof of lemma \ref{lem3.4}).  Let $n=u+\e v,\ m=u-\e v$,  where $\e^2=-1$, then $n^2=m^2=0$ but $n+m=2u$ is not nilpotent, a contradiction. Therefore, the variety  $\V$ is nonmatrix.

 \smallskip
 $(1)\Rightarrow (7)$.  Let an algebra $A$ from a nonmatrix variety $\V$ be generated by a finite set of algebraic elements;   prove that  $A$ is finite dimensional. It suffices to prove that  $A$ coincides with its locally finite radical $L(A)$.  Consider the quotient algebra  $\bar A=A/L(A)$.  In view of \cite[theorem 7]{ZS}, the algebra $\bar A$ is isomorphic to a subdirect sum of prime algebras  $A_{\a}$ without nonzero locally finite ideals.  In particulr, the algebras $A_{\a}$ do not contain nonzero locally nilpotent ideals and by corollary  \ref{cor3.1}  are strongly prime. By proposition  \ref{prop3.2}, all the algebras $A_{\a}$ are commutative and associative, hence so is the algebra $\bar A$. Since the algebra $\bar A$ is generated by a finite set of algebraic elements, it is clearly finite dimensional, that is,  $\bar A=L(\bar A)$.
But by the properties of the radical,  $L(\bar A)=L(A/L(A))=0$, hence $A=L(A)$, and $A$ is finite dimensional.

\smallskip 
 $(7)\Rightarrow (6)$. Assume that the variety $\V$ contains an algebra  $A$ with nilpotent elements  $n,m\in A$ for which the sum  $a=n+m$  is not nilpotent. Consider the algebra $\tilde A=F[x]\otimes_FA$. In view of the infiniteness of the field  $F$, the algebra $\tilde A$ also belongs to $\V$; moreover, the elements of $x\otimes n,\,x\otimes m\in\tilde A$ are nilpotent and hence algebraic over $F$. At the same time, their sum $x\otimes a$ is a transcendent element over $F$. In particular, the subalgebra generated by the elements  $x\otimes n,\,x\otimes m$ is infinite dimensional, which contradicts condition  (7).
 
 \ctd
 
\section{Open questions}

In conclusion, we formulate some open questiones. 

The first question goes back to  \cite{Sh1971}.
\begin{question}\label{q.1}
Let $\V$ be the variety of noncommutative Jordan algebras defined by the identity  $([x,y],y,y)=0$.  Is  $\V$ locally admissible?
\end{question}
It is easy to see that this question is equivalent to the following:
\begin{question}\label{q.2}
Is it true that an anticommutative  3-engelian algebra is locally nilpotent?
\end{question}
\begin{question}
Let $\V$ be a variety of anticommutative algebras such that the free algebra $F_{\V}(n)$ of a given  rank $n\geq 2$ in $\V$ is nilpotent.  Is $\V$ locally nilpotent?
\end{question}
\begin{question}\label{q.3}
Let  $A_n$  be a simple quadratic algebra of dimension $2^n$ obtained by the Cayley-Dickson process. Does  $A_n$ generate a locally admissible variety? For $n\leq 3$ the answer is positive.
\end{question}
\begin{question}\label{q.31}
Describe simple algebras in admissible varieties defined in Corollary \ref{cor3.0}.
\end{question}
\begin{question}\label{q.4}
Can a (locally) admissible variety contain a simple Kokoris algebra which is not a field? Since a simple locally admissible algebra is strongly prime, by theorem \ref{thm3.4} it is possible only in the case of positive characteristic.
\end{question}
\begin{question}\label{q.6}
Is a nonmatrix (locally) admissible variety over a field of characteristic 0 spechtian?
\end{question}
\begin{question}\label{q.7}
Describe simple finite dimensional superalgebras whose Grassmann envelopes generate nonmatrix admissible varieties (see \cite{Kem}). What are identities of these varieties? The last question is open even for Jordan algebras.
\end{question}
The last question is related with the conjecture of the first author on the relation between the existence of the locally nilpotent radical in a variety $\V$ and the property of locally finiteness of $\V$-coalgebras (see \cite{SMS1,SMS2}).
\begin{question}\label{q.8}
Let $\V$ be a locally admissible variety of algebras. Is the theorem on locally finiteness of $\V$-coalgebras true in $\V$?
\end{question}

\end{document}